\documentclass{amsart}

\input xy
\xyoption{all}

\usepackage{amsmath}
\usepackage{amssymb}
\usepackage{hyperref}
\usepackage[final]{showkeys}
\usepackage{amsthm}  

\newtheorem{same}{This should never appear}[section]
\newtheorem{defin}[same]{Definition}

\newtheorem{remark}[same]{Remark}
\newtheorem{theorem}[same]{Theorem}

\newtheorem{lemma}[same]{Lemma}
\newtheorem{fact}[same]{Fact}
\newtheorem{question}[same]{Question}
\newtheorem{cor}[same]{Corollary}
\newtheorem{prop}[same]{Proposition}

\newbox\noforkbox \newdimen\forklinewidth
\setbox0\hbox{$\textstyle\smile$}
\setbox1\hbox to \wd0{\hfil\vrule width \forklinewidth depth-2pt
 height 10pt \hfil}
\setbox\noforkbox\hbox{\lower 2pt\box1\lower 2pt\box0\relax}
\def\unionstick{\mathop{\copy\noforkbox}\limits}

\def\nonfork_#1{\unionstick_{\textstyle #1}}

\setbox0\hbox{$\textstyle\smile$}
\setbox1\hbox to \wd0{\hfil{\sl /\/}\hfil}
\setbox2\hbox to \wd0{\hfil\vrule height 10pt depth -2pt width
              \forklinewidth\hfil}
\newbox\doesforkbox
\setbox\doesforkbox\hbox{\lower 2pt\box1 \lower 2pt\box2\lower2pt\box0\relax}
\def\nunionstick{\mathop{\copy\doesforkbox}\limits}

\def\fork_#1{\nunionstick_{\textstyle #1}}

\newbox\noforkbox \newdimen\forklinewidth
\forklinewidth=0.3pt \setbox0\hbox{$\textstyle\smile$}
\setbox1\hbox to \wd0{\hfil\vrule width \forklinewidth depth-2pt
 height 10pt \hfil}
\wd1=0 cm \setbox\noforkbox\hbox{\lower 2pt\box1\lower
2pt\box0\relax}
\def\unionstick{\mathop{\copy\noforkbox}\limits}
\newcommand{\nf}{\unionstick}
\newcommand{\nfs}[4]{#2 \nf_{#1}^{#4} #3}

\def\1nf{\unionstick^{(1)}}

\def\2nf{\unionstick^{(2)}}
\def\3nf{\unionstick^{(3)}}

\newcommand{\gtp}{\operatorname{gtp}}

\newcommand{\gS}{\operatorname{gS}}

\newcommand{\Sbs}{\gS^\text{bs}}
\newcommand{\gSbs}{\Sbs}

\newcommand{\ba}{\bold{a}}
\newcommand{\bb}{\bold{b}}

\newcommand{\cf}[1]{\operatorname{cf}(#1)}
\newcommand{\rest}{\upharpoonright}

\newcommand{\Ll}{\mathbb{L}}

\newcommand{\LS}{\operatorname{LS}}

\newcommand{\K}{\mathbf{K}}

\newcommand{\leap}[1]{\le_{#1}}
\newcommand{\ltap}[1]{<_{#1}}

\newcommand{\geap}[1]{\ge_{#1}}

\newcommand{\leanf}{\leap{\K_{\lambda^+}}^{\NF}}
\newcommand{\leanfn}{\leap{\K_{\lambda^+}^n}^{\NF}}
\newcommand{\lean}{\leap{\K^n}}

\newcommand{\lta}{\ltap{\K}}
\newcommand{\lea}{\leap{\K}}

\newcommand{\gea}{\geap{\K}}

\newcommand{\NF}{\operatorname{NF}}
\newcommand{\NFhat}{\widehat{\operatorname{NF}}}

\newcommand{\s}{\mathfrak{s}}
\newcommand{\ts}{\mathfrak{t}}

\newcommand{\seq}[1]{\langle #1 \rangle}

\newcommand{\goodp}{\text{good}^+}

\newcommand{\tp}{\textrm{tp}}

\newcommand{\supp}{\text{supp}}

\title{Good frames in the Hart-Shelah example}
\author{Will Boney}
\email{wboney@math.harvard.edu}
\urladdr{http://math.harvard.edu/\textasciitilde wboney/}
\address{Department of Mathematics, Harvard University, Cambridge, MA, USA}
\thanks{This material is based upon work done while the first author was supported by the National Science Foundation under Grant No. DMS-1402191.}
\date{\today }
\author{Sebastien Vasey}
\email{sebv@math.harvard.edu}
\urladdr{http://math.harvard.edu/\textasciitilde sebv/}
\address{Department of Mathematical Sciences, Carnegie Mellon University, Pittsburgh, PA, USA. Current address: Department of Mathematics, Harvard University, Cambridge, MA, USA}

\begin{document}

\begin{abstract}
  For a fixed natural number $n \ge 1$, the Hart-Shelah example is an abstract elementary class (AEC) with amalgamation that is categorical exactly in the infinite cardinals less than or equal to $\aleph_{n}$.

  We investigate recently-isolated properties of AECs in the setting of this example. We isolate the exact amount of type-shortness holding in the example and show that it has a type-full good $\aleph_{n - 1}$-frame which fails the existence property for uniqueness triples. This gives the first example of such a frame. Along the way, we develop new tools to build and analyze good frames. 
\end{abstract}

\maketitle

\section{Introduction}

In his milestone two-volume book on classification theory for abstract elementary classes (AECs) \cite{shelahaecbook, shelahaecbook2}, Shelah introduces a central definition: good $\lambda$-frames.  These are an axiomatic notion of forking for types of singletons over models of cardinality $\lambda$ (see \cite[II.2.1]{shelahaecbook} or Definition \ref{good-frame-def} here). One can think of the statement ``an AEC $\K$ has a good $\lambda$-frame'' as meaning that $\K$ is well-behaved in $\lambda$, where ``well-behaved'' in this context means something similar to superstability in the context of first-order model theory. With this in mind, a key question is:

\begin{question}[The extension question]
  Assume an AEC $\K$ has a good $\lambda$-frame. Under what conditions does $\K$ (or a subclass of saturated models) have a good $\lambda^+$-frame?
\end{question}

Shelah's answer in \cite[II]{shelahaecbook} involves two dividing lines: the existence property for uniqueness triples, and smoothness of a certain ordering $\leanf$ (see Definitions \ref{weakly-succ-def}, \ref{succ-def}). Shelah calls a good frame satisfying the first property \emph{weakly successful} and a good frame satisfying both properties is called \emph{successful}. Assuming instances of the weak diamond, Shelah shows \cite[II.5.9]{shelahaecbook} that the failure of the first property implies many models in $\lambda^{++}$. In \cite[II.8.7]{shelahaecbook} (see also \cite[7.1.3]{jrsh875}), Shelah shows that if the first property holds, then the failure of the second implies there exists $2^{\lambda^{++}}$ many models in $\lambda^{++}$.

However, Shelah does not give any examples showing that these two properties can fail (this is mentioned as part of the ``empty half of the glass'' in Shelah's introduction \cite[N.4.A(f)]{shelahaecbook}). The present paper investigates these dividing lines in the specific setup of the Hart-Shelah example \cite{hs-example}. For a fixed\footnote{Note that our indexing follows \cite{hs-example} and \cite{bk-hs} rather than \cite{ext-frame-jml}.} $n \in [3, \omega)$, the Hart-Shelah example is an AEC $\K^n$ that is categorical exactly in the interval $[\aleph_0, \aleph_{n-2}]$. It was investigated in details by Baldwin and Kolesnikov \cite{bk-hs} who proved that $\K^n$ has (disjoint) amalgamation, is (Galois) stable exactly in the infinite cardinals less than or equal to $\aleph_{n - 3}$, and is $(<\aleph_0, \leq \aleph_{n-3})$-tame (i.e.\ Galois types over models of size at most $\aleph_{n - 3}$ are determined by their restrictions to finite sets, see Definition \ref{tameness-def}).

  The Hart-Shelah example is a natural place to investigate good frames, since it has good behavior only below certain cardinals (around $\aleph_{n - 3}$). The first author has shown \cite[10.2]{ext-frame-jml} that $\K^n$ has a good $\aleph_k$ frame for any $k \le n - 3$, but cannot have one above since stability is part of the definition of a good frame. Therefore at $\aleph_{n - 3}$, the last cardinal when $\K^n$ has a good frame, the answer to the extension question must be negative, so one of the two dividing lines above must fail, i.e.\ the good frame is \emph{not} successful. The next question is: \emph{which} of these properties fails? We show that the first property must fail: the frame is not weakly successful. In fact, we give several proofs (Theorem \ref{negative-prop}, Corollary \ref{negative-prop-2}). On the other hand, we show that the frames strictly below $\aleph_{n - 3}$ are successful\footnote{While there are no known examples, it is conceivable that there is a good frame that is not successful but can still be extended.}. This follows both from a concrete analysis of the Hart-Shelah example (Theorem \ref{hs-successful}) and from abstract results in the theory of good frames (Theorem \ref{abstract-positive}).

  Regarding the abstract theory, a focus of recent research has been the interaction of locality properties and frames. For example, the first author \cite{ext-frame-jml} (with slight improvements in \cite[6.9]{tame-frames-revisited-jsl}) has shown that amalgamation and \emph{tameness} (a locality property for types isolated by Grossberg and VanDieren \cite{tamenessone}) implies a positive answer to the extension question (in particular, the Hart-Shelah example is \emph{not} $(\aleph_{n - 3}, \aleph_{n - 2})$-tame\footnote{This was already noticed by Baldwin and Kolesnikov using a different argument \cite[6.8]{bk-hs}.}). A relative of tameness is type-shortness, introduced by the first author in \cite[3.2]{tamelc-jsl}: roughly, it says that types of sequences are determined by their restriction to small subsequences. Sufficient amount of type-shortness implies (with a few additional technical conditions) that a good frame is weakly successful \cite[Section 11]{indep-aec-apal}.

  As already mentioned, Baldwin and Kolesnikov have shown that the Hart-Shelah example is $(<\aleph_0, \le \aleph_{n - 3})$-tame (see Fact \ref{bk-fact}); Theorem \ref{shortness-thm} refines their argument to show that $\K^n$ is also $(<\aleph_0, <\aleph_{n - 3})$-type short over models of size less than or equal to $\aleph_{n - 3}$ (i.e.\ types of sequences of length less than $\aleph_{n - 3}$ are determined by their finite restrictions, see Definition \ref{tameness-def}). We prove that this is optimal: the result \emph{cannot} be extended to types of length $\aleph_{n - 3}$ (see Corollary \ref{type-short-negative}).

  We can also improve the aforementioned first author's construction of a good $\aleph_k$-frame (when $k \le n -3$) in the Hart-Shelah example.  The good frame built there is not type-full: forking is only defined for a certain (dense family) of basic types. We prove here that the good frame extends to a type-full one. This uses abstract constructions of good frames due to the second author \cite{ss-tame-jsl} (as well as results of VanDieren on the symmetry property \cite{vandieren-symmetry-apal}) when $k \ge 1$. When $k = 0$ we have to work more and develop new general tools to build good frames (see Section \ref{good-frame-sec}).  Motivated by this abstract work, we can give an explicit description of these type-full good extensions (Proposition \ref{exp-tf-nf}).
  
  The following summarizes our main results:

\begin{theorem}
  Let $n \in [3, \omega)$ and let $\K^n$ denote the AEC induced by the Hart-Shelah example. Then:

    \begin{enumerate}
    \item $\K^n$ is $(<\aleph_0, < \aleph_{n-3})$-type short over $\leq \aleph_{n-3}$-sized models and $(<\aleph_0, \le \aleph_{n - 3})$-tame for $(<\aleph_{n-3})$-length types.
    \item $\K^n$ is \emph{not} $(<\aleph_{n - 3}, \aleph_{n - 3})$-type short over $\aleph_{n - 3}$-sized models.
    \item For any $k \le n - 3$, there exists a unique type-full good $\aleph_k$-frame $\s$ on $\K^n$. Moreover:
      \begin{enumerate}
      \item If $k < n - 3$, $\s$ is successful $\goodp$.
      \item If $k = n -3$, $\s$ is \emph{not} weakly successful.
      \end{enumerate}
    \end{enumerate}
\end{theorem}
\begin{proof} \
  \begin{enumerate}
  \item By Theorem \ref{shortness-thm}.
  \item By Corollary \ref{type-short-negative}.
  \item By Theorems \ref{abstract-positive} and Corollary \ref{hs-aleph0-frame}. Note also that by canonicity (Fact \ref{canon-fact}), $\s$ is unique, so extends $\s^{k, n}$ (see Definition \ref{simple-frame-def}).
    \begin{enumerate}
    \item By Theorem \ref{hs-successful}, $\s^{k, n}$ is successful. By Lemma \ref{goodp-lem}, $\s^{k, n}$ is $\goodp$. Now apply Facts \ref{canon-fact} and \ref{type-full-succ}.
    \item By Proposition \ref{negative-prop}, $\s^{k, n}$ is not weakly successful and since $\s$ extends $\s^{k, n}$, $\s$ is not weakly successful either.
    \end{enumerate}
  \end{enumerate}
\end{proof}

We discuss several open questions. First, one can ask whether the aforementioned second dividing line can fail:

\begin{question}[See also {\cite[7.1]{jarden-tameness-apal}}]
  Is there an example of a good $\lambda$-frame that is weakly successful but not successful? 
\end{question}

Second, one can ask whether there is any example at all of a good frame where the forking relation can be defined only for certain types\footnote{After the initial circulation of this paper in July 2016, the second author found that an example of Shelah \cite[VII.5.7]{shelahaecbook2} has a good frame that cannot be extended to be type-full.}:

\begin{question}
  Is there an example of a good $\lambda$-frame that does \emph{not} extend to a type-full frame?
\end{question}

We have not discussed $\goodp$ in our introduction: it is a technical property of frames that allows one to extend frames without changing the order (see the background in Section \ref{prelim-abstract}). No negative examples are known.

\begin{question}
  Is there an example of a good $\lambda$-frame that is not $\goodp$? Is there an example that is successful but not $\goodp$?
\end{question}

In a slightly different direction, we also do not know of an example of a good frame failing symmetry:

\begin{question}[See also {\cite[4.13]{vv-symmetry-transfer-afml}}]
  Is there an example of a triple $(\K, \nf, \Sbs)$ satisfying all the requirements from the definition of a good $\lambda$-frame except symmetry?
\end{question}

In the various examples, the proofs of symmetry either uses disjoint amalgamation (as in \cite[II.3.7]{shelahaecbook}) or failure of the order property (see e.g.\ \cite[5.14]{bgkv-apal}). Recently the second author \cite[4.8]{categ-saturated-afml} has shown that symmetry follows from (amalgamation, no maximal models, and) \emph{solvability} in any $\mu > \lambda$ (see \cite[IV.1.4(1)]{shelahaecbook}; roughly it means that the union of a short chain of saturated model of cardinality $\mu$ is saturated, and there is a ``constructible'' witness). We do not know of an example of a good $\lambda$-frame where solvability in every $\mu > \lambda$ fails.

The background required to read this paper is a solid knowledge of AECs (including most of the material in \cite{baldwinbook09}). Familiarity with good frames and the Hart-Shelah example would be helpful, although we have tried to give a self-contained presentation and quote all the black boxes we need.

This paper was written while the second author was working on a Ph.D.\ thesis under the direction of Rami Grossberg at Carnegie Mellon University and he would like to thank Professor Grossberg for his guidance and assistance in his research in general and in this work specifically. The authors would also like to thank the referees for comments that helped improve the presentation of this paper.

\section{Preliminaries: The abstract theory}\label{prelim-abstract}

Everywhere in this paper, $\K$ denotes a fixed AEC (that may or may not have structural properties such as amalgamation or arbitrarily large models). We assume the reader is familiar with concepts such as amalgamation, Galois types, tameness, type-shortness, stability, saturation, and splitting (see for example the first twelve chapters of \cite{baldwinbook09}). Our notation is standard and is described in the preliminaries of \cite{sv-infinitary-stability-afml}.

On tameness and type-shortness, we use the notation from \cite[3.1,3.2]{tamelc-jsl}:

\begin{defin}\label{tameness-def}
  Let $\lambda \ge \LS (\K)$ and let $\kappa, \mu$ be infinite cardinals\footnote{As opposed to the first author's original definition, we allow $\kappa \le \LS (\K)$ by making use of Galois types over sets, see the preliminaries of \cite{sv-infinitary-stability-afml}.}
  \begin{enumerate}
  \item $\K$ is \emph{$(<\kappa, \lambda)$-tame for $\mu$-length types} if for any $M \in \K_\lambda$ and distinct $p,q \in \gS^{\mu} (M)$, there exists $A \subseteq |M|$ with $|A| < \kappa$ such that $p \rest A \neq q \rest A$. When $\mu = 1$ (i.e.\ we are only interested in types of length one), we omit it and just say that $\K$ is $(<\kappa, \lambda)$-tame.
  \item $\K$ is \emph{$(<\kappa, \mu)$-type short over $\lambda$-sized models} if for any $M \in \K_\lambda$ and distinct $p, q \in \gS^{\mu} (M)$, there exists $I \subseteq \mu$ with $|I| < \kappa$ and $p^I \neq q^I$.
  \end{enumerate}

  We similarly define variations such as ``$\K$ is $(<\kappa, \le \mu)$-type short over $\le \lambda$-sized models.
\end{defin}

\subsection{Superstability and symmetry}

We will rely on the following local version of superstability, already implicit in \cite{shvi635} and since then studied in many papers, e.g.\ \cite{vandierennomax, gvv-mlq, indep-aec-apal, bv-sat-afml, gv-superstability-v5-toappear, vandieren-symmetry-apal}. We quote the definition from \cite[10.1]{indep-aec-apal}:

\begin{defin}\label{ss-def}
$\K$ is \emph{$\mu$-superstable} (or \emph{superstable in $\mu$})
if:

  \begin{enumerate}
    \item $\mu \ge \LS (\K)$.
    \item $\K_\mu$ is nonempty, has joint embedding, amalgamation, and no maximal models.
    \item $\K$ is stable in $\mu$.
    \item There are no long splitting chains in $\mu$:

      For any limit ordinal $\delta < \mu^+$, for every sequence $\langle M_i\mid i<\delta\rangle$ of
      models of cardinality $\mu$ with $M_{i+1}$ universal over $M_i$ and for every $p\in\gS(\bigcup_{i < \delta} M_i)$, there exists $i<\delta$ such that $p$ does not $\mu$-split over $M_i$.
\end{enumerate}
\end{defin}

We will also use the concept of symmetry for splitting isolated in \cite{vandieren-symmetry-apal}:

\begin{defin}\label{sym-def}
For $\mu \ge \LS (\K)$, we say that $\K$ has \emph{$\mu$-symmetry} (or \emph{symmetry in $\mu$}) if  whenever models $M,M_0,N\in\K_\mu$ and elements $a$ and $b$  satisfy the conditions (\ref{limit sym cond})-(\ref{last}) below, then there exists  $M^b$  a limit model over $M_0$, containing $b$, so that $\gtp(a/M^b)$ does not $\mu$-split over $N$.
\begin{enumerate} 
\item\label{limit sym cond} $M$ is universal over $M_0$ and $M_0$ is a limit model over $N$.
\item\label{a cond}  $a\in |M|\backslash |M_0|$.
\item\label{a non-split} $\gtp(a/M_0)$ is non-algebraic and does not $\mu$-split over $N$.
\item\label{last} $\gtp(b/M)$ is non-algebraic and does not $\mu$-split over $M_0$. 
\end{enumerate}
\end{defin}

By an argument of Shelah and Villaveces \cite[2.2.1]{shvi635} (see also \cite{shvi-notes-apal}), superstability holds below a categoricity cardinal.

\begin{fact}[The Shelah-Villaveces Theorem]\label{shvi}
  Let $\lambda > \LS (\K)$. Assume that $\K_{<\lambda}$ has amalgamation and no maximal models. If $\K$ has arbitrarily large models and is categorical in $\lambda$, then $\K$ is superstable in any $\mu \in [\LS (\K), \lambda)$.
\end{fact}
\begin{remark}
  We will only use the result when $\lambda$ is a successor (in fact $\lambda = \mu^+$, where $\mu$ is the cardinal where we want to derive superstability). In this case there is an easier proof due to Shelah. See \cite[I.6.3]{sh394} or \cite[15.3]{baldwinbook09}.
\end{remark}

VanDieren \cite{vandieren-symmetry-apal} has shown that (in an AEC with amalgamation and no maximal models) symmetry in $\mu$ follows from categoricity in $\mu^+$. This was improved in \cite[7.3]{vv-symmetry-transfer-afml} and recently in \cite[4.8]{categ-saturated-afml}, but we will only use VanDieren's original result:

\begin{fact}\label{sym-categ}
  If $\K$ is $\mu$-superstable and categorical in $\mu^+$, then $\K$ has symmetry in $\mu$.
\end{fact}

\subsection{Good frames}

Good $\lambda$-frames were introduced by Shelah in \cite[II]{shelahaecbook} as a bare-bone axiomatization of superstability. We give a simplified definition here:

\begin{defin}[{\cite[II.2.1]{shelahaecbook}}]\label{good-frame-def}
  A \emph{good $\lambda$-frame} is a triple $\s = (\K_\lambda, \nf, \Sbs)$ where:

\begin{enumerate}
\item $\K$ is an AEC such that:
  \begin{enumerate}
  \item $\lambda \ge \LS (\K)$.
  \item $\K_\lambda \neq \emptyset$.
  \item $\K_\lambda$ has amalgamation, joint embedding, and no maximal models.
  \item $\K$ is stable\footnote{In Shelah's original definition, only the set of basic types is required to be stable. However full stability follows, see \cite[II.4.2]{shelahaecbook}.} in $\lambda$.
  \end{enumerate}
  \item For each $M \in \K_\lambda$, $\Sbs (M)$ (called the set of \emph{basic types} over $M$) is a set of nonalgebraic Galois types over $M$ satisfying the \emph{density property}: if $M \lta N$ are both in $\K_\lambda$, there exists $a \in |N| \backslash |M|$ such that $\gtp (a / M; N) \in \Sbs (M)$.
  \item $\nf$ is an (abstract) independence relation on the basic types satisfying invariance, monotonicity, extension existence, uniqueness, continuity, local character, and symmetry (see \cite[II.2.1]{shelahaecbook} for the full definition of these properties).
\end{enumerate}

We say that $\s$ is \emph{type-full} \cite[III.9.2(1)]{shelahaecbook} if for any $M \in \K_\lambda$, $\Sbs (M)$ is the set of \emph{all} nonalgebraic types over $M$. Rather than explicitly using the relation $\nf$, we will say that $\gtp (a / M; N)$ \emph{does not fork} over $M_0$ if $\nfs{M_0}{a}{M}{N}$ (this is well-defined by the invariance and monotonicity properties). We say that a good $\lambda$-frame $\s$ is \emph{on $\K$} if the underlying AEC of $\s$ is $\K_\lambda$, and similarly for other variations.
\end{defin}

\begin{remark}
  We will \emph{not} use the axiom (B) \cite[II.2.1]{shelahaecbook} requiring the existence of a superlimit model of size $\lambda$. In fact many papers (e.g.\ \cite{jrsh875}) define good frames without this assumption. Further, we gave a shorter list of properties that in Shelah's original definition, but the other properties follow, see \cite[II.2]{shelahaecbook}. 
\end{remark}

The next technical property is of great importance in Chapter II and III of \cite{shelahaecbook}. The definition below follows \cite[4.1.5]{jrsh875}.

\begin{defin}\label{weakly-succ-def} Let $\lambda \ge \LS (\K)$.
  \begin{enumerate}
  \item\index{amalgam} For $M_0 \lea M_\ell$ all in $\K_\lambda$, $\ell = 1,2$, an \emph{amalgam of $M_1$ and $M_2$ over $M_0$} is a triple $(f_1, f_2, N)$ such that $N \in \K_\lambda$ and $f_\ell : M_\ell \xrightarrow[M_0]{} N$.
  \item\index{equivalence of amalgam} Let $(f_1^x, f_2^x, N^x)$, $x = a,b$ be amalgams of $M_1$ and $M_2$ over $M_0$. We say $(f_1^a, f_2^a, N^a)$ and $(f_1^b, f_2^b, N^b)$ are \emph{equivalent over $M_0$} if there exists $N_\ast \in \K_\lambda$ and $f^x : N^x \rightarrow N_\ast$ such that $f^b \circ f_1^b = f^a \circ f_1^a$ and $f^b \circ f_2^b = f^a \circ f_2^a$, namely, the following commutes:

  \[
  \xymatrix{ & N^a \ar@{.>}[r]^{f^a} & N_\ast \\
    M_1 \ar[ru]^{f_1^a} \ar[rr]|>>>>>{f_1^b} & & N^b \ar@{.>}[u]_{f^b} \\
    M_0 \ar[u] \ar[r] & M_2 \ar[uu]|>>>>>{f_2^a}  \ar[ur]_{f_2^b} & \\
  }
  \]

  Note that being ``equivalent over $M_0$'' is an equivalence relation (\cite[4.3]{jrsh875}).
\item Let $\s = (\K_\lambda, \nf, \Sbs)$ be a good $\lambda$-frame on $\K$.
  \begin{enumerate}
    \item A triple $(a, M, N)$ is a \emph{uniqueness triple} (for $\s$) if $M \lea N$ are both in $\K_\lambda$, $a \in |N| \backslash |M|$, $\gtp (a / M; N) \in \Sbs (M)$, and for any $M_1 \gea M$ in $\K_\lambda$, there exists a \emph{unique} (up to equivalence over $M$) amalgam $(f_1, f_2, N_1)$ of $N$ and $M_1$ over $M$ such that $\gtp (f_1 (a) / f_2[M_1] ; N_1)$ does not fork over $M$.
    \item $\s$ has the \emph{existence property for uniqueness triples} (or is \emph{weakly successful}) if for any $M \in \K_{\lambda}$ and any $p \in \Sbs (M)$, one can write $p = \gtp (a / M; N)$ with $(a, M, N)$ a uniqueness triple.
  \end{enumerate}
  \end{enumerate}
\end{defin}

The importance of the existence property for uniqueness triples is that it allows us to extend the nonforking relation to types of models (rather than just types of length one). This is done by Shelah in \cite[II.6]{shelahaecbook} but was subsequently simplified in \cite{jrsh875}, so we quote from the latter.

\begin{defin}\label{nf-def}
  Let $\s$ be a weakly successful good $\lambda$-frame on $\K$, with $\K$ categorical in $\lambda$.

  \begin{enumerate}
  \item\cite[5.3.1]{jrsh875} Define a 4-ary relation $\NF^\ast = \NF_{\s}^\ast$ on $\K_\lambda$ by $\NF^\ast (N_0, N_1, N_2, N_3)$ if there is $\alpha^\ast < \lambda^+$ and for $\ell = 1,2$ there are increasing continuous sequences $\seq{N_{\ell, i} : i \le \alpha^\ast}$ and a sequence $\seq{d_i : i < \alpha^\ast}$ such that:
    \begin{enumerate}
    \item $\ell < 4$ implies $N_0 \lea N_{\ell} \lea N_3$.
    \item $N_{1,0} = N_0$, $N_{1, \alpha^\ast} = N_1$, $N_{2, 0} = N_2$, $N_{2, \alpha^\ast} = N_3$.
    \item $i \le \alpha^\ast$ implies $N_{1, i} \lea N_{2, i}$.
    \item $d_i \in |N_{1, i + 1}| \backslash |N_{1, i}|$.
    \item $(d_i, N_{1, i}, N_{1, i + 1})$ is a uniqueness triple.
    \item $\gtp (d_i / N_{2, i}; N_{2, i + 1})$ does not fork over $N_{1, i}$.
    \end{enumerate}
  \item \cite[5.3.2]{jrsh875} Define a 4-ary relation $\NF = \NF_{\s}$ on $\K_\lambda$ by $\NF (M_0, M_1, M_2, M_3)$ if there are models $N_0, N_1, N_2, N_3$ such that $N_0 = M_0$, $\ell < 4$ implies $M_\ell \lea N_\ell$ and $\NF^\ast (N_0, N_1, N_2, N_3)$.
  \end{enumerate}
\end{defin}

By \cite[5.5.4]{jrsh875}, $\NF$ satisfies several of the basic properties of forking:

\begin{fact}\label{nf-fact}
  If $\NF (M_0, M_1, M_2, M_3)$, then $M_1 \cap M_2 = M_0$. Moreover, $\NF$ respects $\s$ and satisfies monotonicity, existence, weak uniqueness, symmetry, and long transitivity (see \cite[5.2.1]{jrsh875} for the definitions).
\end{fact}

Shelah \cite[III.1.1]{shelahaecbook} says a weakly successful good frame is \emph{successful} if an ordering $\leanf$ defined in terms of the relation $\NF$ induces an AEC. We quote the full definition from \cite{jrsh875}.

\begin{defin}\label{succ-def}
  Let $\s$ be a weakly successful good $\lambda$-frame on $\K$, with $\K$ categorical in $\lambda$.
  \begin{enumerate}
  \item\cite[6.1.2]{jrsh875} Define a 4-ary relation $\NFhat = \NFhat_{\s}$ on $\K$ by $\NFhat (N_0, N_1, M_0, M_1)$ if:
    \begin{enumerate}
    \item $\ell < 2$ implies that $N_\ell \in \K_\lambda$, $M_\ell \in \K_{\lambda^+}$.
    \item There is a pair of increasing continuous sequences $\seq{N_{0, \alpha} : \alpha < \lambda^+}$, $\seq{N_{1, \alpha} : \alpha \le \lambda^+}$ such that for every $\alpha < \lambda^+$, $\NF (N_{0, \alpha}, N_{1, \alpha}, N_{0, \alpha + 1}, N_{1, \alpha + 1})$ and for $\ell < 2$, $N_{0, \ell} = N_\ell$, $M_\ell = N_{\ell, \lambda^+}$.
    \end{enumerate}
  \item \cite[6.1.4]{jrsh875} For $M_0 \lea M_1$ both in $\K_{\lambda^+}$, $M_0 \leanf M_1$ if there exists $N_0, N_1 \in \K_\lambda$ such that $\NFhat (N_0, N_1, M_0, M_1)$.
  \item \cite[10.1.1]{jrsh875} $\s$ is \emph{successful} if $\leanf$ satisfies smoothness on the saturated models in $\K_{\lambda^+}$: whenever $\delta < \lambda^{++}$ is limit, $\seq{M_i : i \le \delta}$ is a $\leanf$-increasing continuous sequence of saturated models of cardinality $\lambda^+$, and $N \in \K_{\lambda^+}$ is saturated such that $i < \delta$ implies $M_i \leanf N$, then $M_{\delta} \leanf N$.
  \end{enumerate}
\end{defin}

The point of successful good frames is that they can be extended to a good $\lambda^+$-frame on the class of saturated model of cardinality $\lambda^+$ (see \cite[10.1.9]{jrsh875}). The ordering of the class will be $\leanf$. Shelah also defines what it means for a frame to be $\goodp$. If the frame is successful, being $\goodp$ implies that $\leanf$ is just $\lea$ and simplifies several arguments \cite[III.1.3, III.1.8]{shelahaecbook}:

\begin{defin}\label{goodp-def}
  A good $\lambda$-frame $\s$ on $\K$ is \emph{$\goodp$} when the following is \emph{impossible}:

  There exists an increasing continuous $\seq{M_i : i < \lambda^+}$, $\seq{N_i : i < \lambda^+}$, a basic type $p \in \gS (M_0)$, and $\seq{a_i : i < \lambda^+}$ such that for any $i < \lambda^+$:

  \begin{enumerate}
  \item $i < \lambda^+$ implies that $M_i \lea N_i$ and both are in $\K_\lambda$.
  \item $a_{i + 1} \in |M_{i + 2}|$ and $\gtp (a_{i + 1} / M_{i + 1}; M_{i + 2})$ is a nonforking extension of $p$, but $\gtp (a_{i + 1} / N_0; N_{i + 2})$ is not.
  \item $\bigcup_{j < \lambda^+} M_j$ is saturated.
  \end{enumerate}
\end{defin}

\begin{fact}
  Let $\s$ be a successful good $\lambda$-frame on $\K$. The following are equivalent:

  \begin{enumerate}
  \item\label{succ-equiv-1} $\s$ is $\goodp$.
  \item\label{succ-equiv-2} For $M, N \in \K_{\lambda^+}$ both saturated, $M \leanf N$ if and only if $M \lea N$.
  \end{enumerate}
\end{fact}
\begin{proof}
  (\ref{succ-equiv-1}) implies (\ref{succ-equiv-2}) is \cite[III.1.8]{shelahaecbook}. Let us see that (\ref{succ-equiv-2}) implies (\ref{succ-equiv-1}): Suppose for a contradiction that $\seq{M_i : i < \lambda^+}$, $\seq{N_i : i < \lambda^+}$, $p$, $\seq{a_i : i < \lambda^+}$ witness that $\s$ is \emph{not} $\goodp$. Write $M_{\lambda^+} := \bigcup_{i < \lambda^+} M_i$, $N_{\lambda^+} := \bigcup_{i < \lambda^+} N_i$. Using \cite[6.1.6]{jrsh875}, we have that there exists a club $C \subseteq \lambda^+$ such that for any $i < j$ both in $C$, $\NF (M_i, M_j, N_i, N_j)$. In particular (by monotonicity), $\NF (M_i, M_{i + 2}, N_i, N_{i + 2})$. Pick any $i \in C$. Because $\NF$ respects $\s$ (Fact \ref{nf-fact}), $\gtp (a_{i + 1} / N_i; N_{i + 2})$ does not fork over $M_i$. By the properties of $\seq{a_i : i < \lambda^+}$, $\gtp (a_{i + 1} / M_{i + 1}; M_{i + 2})$ is a nonforking extension of $p$. By transitivity, $\gtp (a_{i + 1} / N_i; N_{i + 2})$ also is a nonforking extension of $p$, contradicting the definition of $\goodp$.
\end{proof}

\begin{fact}[{\cite[III.1.8]{shelahaecbook}}]\label{succ-fact}
  Let $\s$ be a successful $\goodp$ $\lambda$-frame on $\K$. Then there exists a good $\lambda^+$-frame $\s^+$ with underlying AEC the saturated models in $\K$ of size $\lambda^+$ (ordered with the strong substructure relation inherited from $\K$).
\end{fact}

We will also use that successful $\goodp$ frame can be extended to be type-full.

\begin{fact}[{\cite[III.9.6(2B)]{shelahaecbook}}]\label{type-full-succ}
  If $\s$ is a successful $\goodp$ $\lambda$-frame on $\K$ and $\K$ is categorical in $\lambda$, then there exists a \emph{type-full} successful $\goodp$ $\lambda$-frame $\ts$ with underlying class $\K_\lambda$.
\end{fact}

The next result derives good frames from some tameness and categoricity. The statement is not optimal (e.g.\ categoricity in $\lambda^+$ can be replaced by categoricity in any $\mu > \lambda$) but suffices for our purpose.

\begin{fact}\label{good-frame-constr}
  Assume that $\K$ has amalgamation and arbitrarily large models. Let $\LS (\K) < \lambda$ be such that $\K$ is categorical in both $\lambda$ and $\lambda^+$. Let $\kappa \le \LS (\K)$ be an infinite regular cardinal such that $\LS (\K) = \LS (\K)^{<\kappa}$ and $\lambda = \lambda^{<\kappa}$.

  If $\K$ is $(\LS (\K), \le \lambda)$-tame, then there is a type-full good $\lambda$-frame $\s$ on $\K$. If in addition $\K$ is $(\LS (\K), \le \lambda)$-tame for $(<\kappa)$-length types and $(<\kappa, \le \lambda)$-type-short over $\lambda$-sized models, then $\s$ is weakly successful.
\end{fact}
\begin{proof}
  By Facts \ref{shvi} and \ref{sym-categ}, $\K$ is superstable in any $\mu \in [\LS (\K), \lambda]$, and has $\lambda$-symmetry. By \cite[6.4]{vv-symmetry-transfer-afml}, there is a type-full good $\lambda$-frame $\s$ on $\K_\lambda$. For the last sentence is by \cite[3.13]{downward-categ-tame-apal}.
\end{proof}

Fact \ref{good-frame-constr} gives a criteria for when a good frame is \emph{weakly} successful, but when is it successful? This is answered by the next result, due to Adi Jarden \cite[7.19]{jarden-tameness-apal} (note that the conjugation hypothesis there follows from \cite[III.1.21]{shelahaecbook}).

\begin{fact}\label{succ-criteria}
  Let $\s$ be a weakly successful good $\lambda$-frame on $\K$. If $\K$ is categorical in $\lambda$, has amalgamation in $\lambda^+$, and is $(\lambda, \lambda^+)$-tame, then $\s$ is successful $\goodp$.
\end{fact}

We will also make use of the following result, which tells us that if the AEC is categorical, there can be at most one good frame \cite[9.7]{indep-aec-apal}:

\begin{fact}[Canonicity of categorical good frames]\label{canon-fact}
  Let $\s$ and $\ts$ be good $\lambda$-frame on $\K$ with the same basic types. If $\K$ is categorical in $\lambda$, then $\s = \ts$.
\end{fact}

\section{Preliminaries: Hart-Shelah}
\begin{defin}
  Fix $n \in [2, \omega)$.  Let $\K^n$ be the AEC from the Hart-Shelah example.  This class is $\Ll_{\omega_1, \omega}$-definable and a model in $\K^n$ consists of the following:
\begin{itemize}
	\item $I$, some arbitrary index set
	\item $K = [I]^n$ with a membership relation for $I$
	\item $H$ is a copy of $\mathbb{Z}_2$ with addition
	\item $G = \oplus_{u \in K} \mathbb{Z}_2$ with the evaluation map from $G \times K$ to $\mathbb{Z}_2$ and functions that indicate the support of $G$
	\item $G^*$ is a set with a projection $\pi_{G^*}$ onto $K$ such that there is a 1-transitive action of $G$ on each stalk $G^*_u = \pi_{G^*}^{-1}(u)$; we denote this action by $t_G(u, \gamma, x, y)$ for $u \in K$, $\gamma \in G$, and $x, y \in G^*_u$
	\item $H^*$ is a set with a projection $\pi_{H^*}$ onto $K$ such that there is a 1-transitive action of $\mathbb{Z}_2$ on each stalk $H^*_u =\pi_{H^*}^{-1}(u)$ denoted by $t_H$
	\item $Q$ is a $(n+1)$-ary relation on $(G^*)^n \times H^*$ satisfying the following:
	\begin{itemize}
		\item We can permute the first $n$ elements (the one from $G^*$) and preserve $Q$ holding.
		\item If $Q(x_1, \dots, x_n, y)$ holds, then the indices of their stalks are \emph{compatible}, which means the following: $x_\ell \in G^*_{u_\ell}$ and $y \in H^*_v$ such that $\{u_1, \dots, u_n, v\}$ are all $n$ element subsets of some $n+1$ element subset of $I$.
		\item $Q$ is preserved by ``even" actions in the following sense: suppose 
		\begin{itemize}
			\item $u_1, \dots, u_n, v \in K$ are compatible
			\item $x_\ell, x_\ell' \in G^*_{u_\ell}$ and $y, y' \in H^*_{v}$
			\item $\gamma_\ell \in G$ and $\ell \in \mathbb{Z}_2$ are the unique elements that send $x_\ell$ or $y$ to $x_\ell'$ or $y'$
		\end{itemize}
		then the following are equivalent
		\begin{itemize}
			\item $Q(x_1, \dots, x_n, y)$ if and only if $Q(x'_1, \dots, x'_n, y')$
			\item $\gamma_1(v) + \dots + \gamma_n(v) + \ell = 0 \mod 2$
		\end{itemize}
	\end{itemize}
\end{itemize}

For $M, N \in \K^n$, $M \lean N$ if and only if $M \prec_{\Ll_{\omega_1, \omega}} N$.


\end{defin}

\begin{fact}[\cite{bk-hs}]\label{bk-fact}
  Let $n \in [2, \omega)$.

    \begin{enumerate}
    \item $\K^n$ has disjoint amalgamation, joint embedding, and arbitrarily large models.
    \item \label{mod-comp} $\K^n$ is model-complete: For $M, N \in \K^n$, $M \lean N$ if and only if $M \subseteq N$.
    \item For any infinite cardinal $\lambda$, $\K^n$ is categorical in $\lambda$ if and only if $\lambda \le \aleph_{n - 2}$.
    \item $\K^n$ is not stable in any $\lambda \ge \aleph_{n - 2}$.
    \item If $n \ge 3$, then $\K^n$ is $(<\aleph_0, \le \aleph_{n - 3})$-tame, but it is \emph{not} $(\aleph_{n - 3}, \aleph_{n - 2})$-tame.
    \end{enumerate}
\end{fact}

A crucial point for (\ref{mod-comp}) is that the language computes the support of the functions in $G(M)$, so that the supports cannot grow as the model does; such substructures are called \emph{full} in \cite{bk-hs}.  Note that the entire universe of a model of $\K^n$ is determined by the index $I$, so if $M \subsetneq N$, then $I(M) \subsetneq I(N)$. Thus it is natural to define a frame whose basic types are just the types of elements in $I$ and nonforking is just nonalgebraicity. The following definition appears in the proof of \cite[10.2]{ext-frame-jml}:

\begin{defin}\label{simple-frame-def}
  Let $n \in [3, \omega)$. For $k \le n - 3$, let $\s^{k, n} = (\K_{\aleph_k}^n, \nf, \Sbs)$ be defined as follows:

    \begin{itemize}
    \item $p \in \gSbs (M)$ if and only if $p = \gtp (a / M; N)$ for $a \in I (N) \backslash I (M)$.
    \item $\gtp(a/M_1; M_2)$ does not fork over $M_0$ if and only if $a \in I(M_2) \backslash I(M_1)$.
    \end{itemize}
\end{defin}
\begin{remark}
  By \cite[10.2]{ext-frame-jml}, $\s^{k, n}$ is a good $\aleph_k$-frame.  We will extend this to a type-full good frame in Theorem \ref{abstract-positive}.
\end{remark}

The notion of a solution is key to analyzing models of $\K^n$. 

\begin{defin}[{\cite[2.1 and 2.3]{bk-hs}}]
Let $M \in \K^n$.
\begin{enumerate}
	\item $h = (f,g)$ is a solution for $W \subseteq K(M)$ if and only if $f \in \Pi_{u \in W} G^*_u(M)$ and $g \in \Pi_{u \in W} H^*_u(M)$ such that, for all compatible $u_1, \dots, u_n, v \in W$, we have
	$$M \vDash Q\left( f(u_1), \dots. f(u_n), g(v)\right)$$
	\item $h= (f, g)$ is a solution over $A \subseteq I(M)$ if and only if it is a solution for $[A]^n$.
	\item $h=(f, g)$ is a solution for $M$ if and only if it is a solution for $K(M)$.
\end{enumerate}
Given $f:M \cong N$ and solutions $h^M$ for $M$ and $h^N$, we say that $h^M$ and $h^N$ are conjugate by $f$ if
$$f^N = f \circ f^M \circ f^{-1} \text{ and } g^N = f \circ g^M \circ f^{-1}$$
We write this as $h^N = f \circ h^M \circ f^{-1}$.
\end{defin}

A key notion is that of extending and amalgamating solutions.

\begin{defin}[{\cite[2.9]{bk-hs}}] \
\begin{enumerate}
	\item A solution $h = (f, g)$ extends another solution $h' = (f', g')$ if $f' \subseteq f$ and $g' \subseteq g$.
	\item We say that $\K^n$ has \emph{$k$-amalgamation for solutions over sets of size $\lambda$} if given any $M \in \K^n$, $A \subseteq I(M)$ of size $\lambda$, $\{b_1, \dots, b_n\} \subseteq I(M)$, and solutions $h_w$ over $A \cup\{ b_i \mid i \in w\}$ for every $w \in [\{b_1, \dots, b_b\}]^{n-1}$ such that $\bigcup_{w} h_w$ is a function, there is a solution $h$ for $A \cup \{b_i \mid i \leq n\}$ that extends all $h_w$.
\end{enumerate}
\end{defin}

0-amalgamation is often referred to simply as the existence of solutions and 1-amalgamation is the extension of solutions.

Forgetting the $Q$ predicate, $M \in \K^n$ is a bunch of affine copies of $G^M$, so an isomorphism is determined by a bijection between the copies and picking a 0 from each affine copy.  However, adding $Q$ complicates this picture.  Solutions are the generalization of picking 0's to $\K^n$.  Thus, amongst the models of $\K^n$ admitting solutions (which is at least $\K^n_{\aleph_{n-2}}$, see Fact \ref{amal-sol-fact}), there is a strong, functorial correspondence between isomorphisms between $M$ and $N$ and pairs of solutions for $M$ and $N$.

The following is implicit in \cite{bk-hs}, see especially Lemma 2.6 there.

\begin{fact}\label{sol-iso-corr}
We work in $\K^n$.
\begin{enumerate}
	\item Given $f:M \cong N$ and a solution $h^M$ of $M$, there is a unique solution $h^N$ of $N$ that is conjugate to $h^M$ by $f$.  Moreover, if $f': M' \cong N'$ extends\footnote{So $M \lean M'$ and $N \lean N'$.} $f$ and $h^{M'}$ is a solution of $M'$ extending $h^M$, then the resulting $h^{N'}$ extends $h^N$.
	\item Given solutions $h^M$ for $M$ and $h^N$ for $N$ and a bijection $h_0:I(M) \to I(N)$, there is a unique isomorphism $f:M \cong N$ extending $h_0$ such that $h^M$ and $h^N$ are conjugate by $f$.  Moreover, if $h^{M'}$ and $h^{N'}$ are solutions for $M'$ and $N'$ that extend $h^M$ and $h^N$, then the resulting $f'$ extends $f$.
	\item These processes are inverses of each other: if we have [$f: M \cong N$ and a solution $h^M$ of $M$]/[solutions $h^M$ and $h^N$ for $M$ and $N$ and a bijection $h_0:I(M) \to I(N)$] and then apply [(1) and then (2)]/[(2) and then (1)], then [the resulting isomorphism is $f$]/[the resulting solutions for $N$ is $h^N$],
\end{enumerate}
\end{fact}




\begin{lemma} \label{not-aff-ext}
Suppose $M, N \in \K^n$ and $f_0:I(M) \to I(N)$ is an injection.  Then there is a unique extension to $f_1$ with domain $M - \left(G^*(M) \cup H^*(M)\right)$ that must be extended by any strong embedding extending $f_0$.
\end{lemma}

\begin{proof}
$M - \left(G^*(M) \cup H^*(M)\right)$ is the definable closure of $I(M)$, so the value of $f_0$ on $I(M)$ determines the value on $M - \left(G^*(M) \cup H^*(M)\right)$.
\end{proof}

For the following, write $\aleph_{-1}$ for finite.

\begin{fact}\label{amal-sol-fact}
  Let $n \in [2, \omega)$, $k_0 < \omega$, and $k_1 \in \{-1\} \cup \omega$. The following are equivalent:
\begin{enumerate}
	\item $\K^n$ has $k_0$-amalgamation of solutions over $\aleph_{k_1}$-sized sets.
	\item $k_0 + k_1 \leq n-2$.
\end{enumerate}
\end{fact}
\begin{proof}
  (1) implies (2) by the examples of \cite[Section 6]{bk-hs}.  (2) implies (1) by combining \cite[2.11, 2.14]{bk-hs}.
\end{proof}

We could do many more variations on the following, but I think this statement suffices for what we need to show.

\begin{defin}\label{standard-def}
  For $n \in [2, \omega)$ and $I$ an index set, the \emph{standard model for $I$} is the unique $M \in \K^n$ such that $G^* (M) = K \times G_K$, where $K := [I]^n$.
\end{defin}

\begin{lemma}\label{ext-imp-stand}
  Let $n \in [3, \omega)$. Given any $M \lean N$ from $\K^n_{\leq \aleph_{n-3}}$, we may assume that they are standard.  That is, if we write $M^*$ for the standard model of $I(M)$ and $N^*$ for the standard model on $I(N)$, then there is an isomorphism $f:N \cong_{I(N)} N^*$ that restricts to an isomorphism $M \cong_{I(M)} M^*$.
\end{lemma}

\begin{proof}
  Find a solution $h^M$ for $M$ and extend it to a solution $h^N$ for $N$; this is possible by Fact \ref{amal-sol-fact} since $(n-3) + 1 \leq n-2$.  We have solutions $h^{M^*}$ and $h^{N^*}$ for $M^*$ and $N^*$ because they are the standard models and, thus, have solutions.  Then Theorem \ref{sol-iso-corr} allows me to build an isomorphism between $M$ and $M^*$ and extend it to $f:N \cong N^*$, each of which extend the identity on $I$.
\end{proof}

\section{Tameness and shortness}

The following is a strengthening of \cite[5.1]{bk-hs} to include type-shortness. 

\begin{theorem}\label{shortness-thm}
  For $n \in [3, \omega)$, $\K^n$ is $(<\aleph_0, < \aleph_{n-3})$-type short over $\leq \aleph_{n-3}$-sized models and $(<\aleph_0, \le \aleph_{n - 3})$-tame for $(<\aleph_{n-3})$-length types.  Moreover, these Galois types are equivalent to first-order existential (syntactic) types.
\end{theorem}
\begin{proof}

For this proof, write $\tp_\exists$ for the first-order existential type.  We prove the type-shortness claim.  The tameness result follows from \cite[5.1]{bk-hs}.

  Let $M \in \K^n_{\leq \aleph_{n-3}}$ and $M \lean N^A, N^B$ with $A \subseteq |N^A|$, $B \subseteq |N^B|$ of size $\leq \aleph_{n-4}$ (we use our convention from Fact \ref{amal-sol-fact} that $\aleph_{-1}$ means finite) such that $\tp_\exists(A/M; N^A) = \tp_\exists(B/M; N^B)$.  By \cite[4.2]{bk-hs}, we can find minimal, full substructures $M^A$ and $M^B$.  Additionally, for each finite $\ba \in A$ and $\bb \in B$, we can find minimal full substructures $M^\ba$ and $M^\bb$ in $M^A$ and $M^B$.  It's easy to see that $M^A$ is the directed union of $\{M^\ba \mid \ba \in A\}$ and similarly for $M^B$; note that we don't necessarily have $M^\ba, M^{\ba'} \subseteq |M^{\ba \cup \ba'}|$, but they are in $M^{M^{\ba} \cup M^{\ba'}}$.
  
Set $M_0 = M^A \cap M$.  We want to build $f_0 : M^A \to_{M_0} N^B$ such that $f_0(A) = B$.  Similarly, construct $M^B$.  Note that 
$$M_0=M^A \cap M = \cup_{\ba \in M} (M^\ba \cap M_0) = \cup_{\bb\in M} (M^\bb \cap M_0) = M^B \cap M_0$$
By assumption, we have $\tp_\exists(A/M_0; M^A) = \tp_\exists(B/M_0; M^A)$.  Set $X = \{ \pi^{M^A}(x) \mid x \in A \cap G^*(M^A)\}$ and $Y = \{ \pi^{M^B}(x) \mid x \in B \cap G^*(M^B)\}$, indexed appropriately.  \\

{\bf Claim:} $\tp_\exists(AX/M_0; M^A) = \tp_\exists(BY/M_0; M^A)$

This is true because all of the added points are in the definable closure via an existential formula.\\

Thus, the induced partial map $f:AX \to BY$ is $\exists$-elementary.  By Fact \ref{amal-sol-fact}, we have extensions of solutions.  Let $h^{M^A}$ be a solution for $M^A$.  Then we can restrict this to $h^X$ which is a solution for $X$.  Then we can define a solution $h^Y$ for $Y$ by conjugating it with $f$.  Finally, we can extend $h^Y$ to a solution $h^{M^B}$ for $M^B$.  Since they satisfy the same existential type and the extensions are minimally constructed, we can define a bijection $h_0:I(M^A) \to I(M^B)$ respecting the type.  Given the two solutions and the bijection $h_0$, we can use Theorem \ref{sol-iso-corr} to find an isomorphism $f_0: M^A \cong M^B$ extending $h_0$ and making these solutions conjugate.  By construction, $f_0$ fixes $M_0$ and sends $A$ to $B$.

Resolve $M$ as $\seq{M_i \mid i < \alpha}$ starting with $M_0$ so $\|M_i\| \leq \aleph_{n-4}$.  Then find increasing continuous $\seq{M^A_i, M^B_i \mid i < \alpha}$ by setting $M^A_0 = M^A$ and $M^A_{i+1}$ to be a disjoint amalgam\footnote{Crucially, it is an amalgam such that $I(M^A_{i+1}) = I(M^A_i) \cup I(M_{i+1})$ with the union disjoint over $I(M_i)$; this is guaranteed by the second clause of the claim.} of $M_{i+1}$ and $M^A_i$ over $M_i$, and similarly for $M^B_i$.  

Using extension of solutions, we can find an increasing chain of solutions $\seq{ h^{M_i} \mid i < \alpha}$ for $M_i$.  Using 2-amalgamation of solutions over $\leq \alpha_{n-4}$ sized sets\footnote{Crucially, this holds here, but fails at the next cardinal.  Thus, we couldn't use this argument to get $(< \aleph_0, \aleph_{n-3})$-type shortness or over $\aleph_{n-2}$ sized models.}, we can find increasing chains of solutions $\seq{h^{M^A_i}, h^{M^B_i} \mid i < \alpha}$ for $M^A_i$ and $M^B_i$, respectively, such that $h^{M^A_i}$ also extends $h^{M_i}$.

By another application of Theorem \ref{sol-iso-corr}.(2), this gives us an increasing sequence of isomorphism $\seq{f_i: M^A_i \cong_{M_i} M^B_i \mid i < \alpha}$; here we are using that $I(M^A_{i+1}) - I(M^A_i) = I(M^B_{i+1})-I(M^B_i)$.  At the top, we have that $f_{\alpha}:M^A \cong_M M^B$.  This demonstrates that $\gtp(A/M; N^A) = \gtp(B/M; N^B)$.
\end{proof}

Baldwin and Kolesnikov \cite{bk-hs} have shown that tameness fails at the next cardinal and we will see later (Corollary \ref{type-short-negative}) that $\K^n$ is \emph{not} $(<\aleph_{n - 3}, \aleph_{n - 3})$-type short over $\aleph_{n - 3}$-sized models. 

\section{What the abstract theory tells us}

We combine the abstract theory with the facts derived so far about the Hart-Shelah example.

We first give an abstract argument that in the Hart-Shelah example good frames below $\aleph_{n - 3}$ are weakly successful (in fact successful):

\begin{theorem}\label{abstract-positive}
  Let $n \in [3, \omega)$. For any $k \in [1, n - 3]$, there is a type-full good $\aleph_k$-frame $\s$ on $\K^n$. Moreover, $\s$ (and therefore $\s^{k, n}$) is successful if $k < n - 3$.
\end{theorem}
\begin{proof}
  Let $\lambda := \aleph_k$. First, assume that $k < n - 3$. By Fact \ref{bk-fact}, $\K^n$ is categorical in $\lambda$, $\lambda^+$ and is $(<\aleph_0, \le \lambda^+)$-tame. By Theorem \ref{shortness-thm}, $\K$ is $(<\aleph_0, \lambda)$-type-short over $\lambda$-sized models. Thus one can apply Fact \ref{good-frame-constr} (where $\kappa$ there stands for $\aleph_0$ here) to get a weakly successful type-full good $\lambda$-frame $\s$ on $\K^n$. By Fact \ref{succ-criteria}, $\s$ is actually successful. This implies that $\s^{k, n}$ is successful by canonicity (Fact \ref{canon-fact}). 
  
  Second, assume $k = n - 3$. We can still apply Fact \ref{good-frame-constr} to get the existence of a type-full good $\lambda$-frame $\s$, although we do not know it will be weakly successful (in fact this will fail, see Proposition \ref{negative-prop}).  Then Fact \ref{canon-fact} implies that $\s^{k, n}$ is $\s$ restricted to types in $I$.
\end{proof}

We can give an explicit description of the type-full frame $\s$ guaranteed to exist by Theorem \ref{abstract-positive}.  First, we give a nice characterization of when a model is universal or limit over another.

\begin{theorem} \label{universal-char-claim}
  Let $n \in [3, \omega)$. Let $k \le n - 3$ and let $M_0, M_1 \in \K_{\aleph_k}^n$. Then $M_1$ is universal over $M_0$ if and only if $|I(M_1) - I(M_0)| = \|M_1\|$. In particular, $M_1$ is universal over $M_0$ if and only if $M_1$ is limit over $M_0$.
\end{theorem}
\begin{proof}
 First suppose that $M_1$ is universal over $M_0$.  We don't have maximal models, so let $M_0 \lean N_*$ be such that $|I(N_*) - I(M_0)| = \|M_1\|$.  We have that $\|N_*\| = \|M_1\|$, so there is an embedding $f:N_* \to_{M_0} M_1$.  Then $f(I(N_*)) \subseteq I(M_1)$.

Now suppose that $|I(M_1) - I(M_0)| = \|M_1\|$ and let $M_0 \lean N_*$ with $\|N_*\| = \|M_1\|$.  Let $I^- \subseteq I(M_1) - I(M_0)$ be of size $|I(N_*) - I(M_0)|$ and let $M^- \lean M_1$ have $I(M^-) = I(M_0) \cup I(M^-)$.  Let $(f, g)$ be a solution for $M_0$.  Since we have extensions of solutions, we can extend this to solutions $(f^-, g^-)$ on $M^-$ and $(f_*, g_*)$ on $N_*$.    The whole point of solutions is that this allows us to build an isomorphism between $M^-$ and $N_*$ over $M_0$ by mapping the solutions to each other (see Theorem \ref{sol-iso-corr}).
\end{proof}

Let $M_0 \subseteq M \subsetneq N \in \K^n_{\aleph_k}$ and $a \in N - M$.  Following the proof of Fact \ref{good-frame-constr} to \cite[6.4]{vv-symmetry-transfer-afml}, the definition of $\s$ is given by
 \begin{center} 
for any $\mu \in [\aleph_0, \aleph_{n-3}]$, $\gtp(a/M; N)$ does not fork over $M_0$ if and only if for some/any $M_0^* \subseteq M_0$ of size $\mu$ such that $|I(M_0) - I(M_0^*)| \geq \mu$, we have $\gtp(a/M;N)$ does not $\mu$-split over $M_0^*$.  
\end{center}
The flexibility on $\mu$ follows from tameness and \cite[6.9]{tame-frames-revisited-jsl}, while the ``some/any" equivalence follows because these cardinals have extension of solutions.

Using that Galois types correspond to existential first-order types (Theorem \ref{shortness-thm}) and other specifics of the example, we can give more explicit descriptions of the nonforking in each sort.  To do so, for $\gamma \in G(N)$, define 
$$\supp_N\gamma := \{i \in I(N) : \exists k \in K(N). N\vDash ``i \in k \wedge \gamma(k) = 1"\}$$
This is the support of $\gamma$ as viewed as a function from $K(N)$ to $\mathbb{Z}_2$.  The structure on $\K^n$ makes this the image of $\gamma$ under certain functions of the language.  In particular, the support cannot grow in any extension and if $N_0 \subseteq N$ and $\supp_N\gamma \subseteq N_0$, then $\gamma \in N_0$.

We can characterize nonforking according to $\s$ along the following lines:
\begin{prop} \label{exp-tf-nf}
 Fix $n \in [3, \omega)$ and $k \leq n - 3$. Let $M_0 \subseteq M \subseteq N \in \K^n_{\aleph_k}$ and $a \in N - M$.
 
 \begin{enumerate}
\item If $a \in I(N)$ or $a \in K(N)$, then $\gtp(a/M; N)$ does not fork over $M_0$ if and only if $a \in N - M$.
\item \label{nf-G} If $\gamma \in G(N)$, then $\gtp(\gamma/M; N)$ does not fork over $M_0$ if and only if $\supp_N\gamma \cap M \subseteq M_0$ and $\gamma \not \in M$.
\item If $a \in G^*(N)$, then $\pi(a) \in K(N)$ is the index of the fiber and:
\begin{enumerate}
	\item If $\pi(a) \in K(N) - K(M)$, then $\gtp(a/M;N)$ does not fork over $M_0$.
	\item If $\pi(a) \in K(M) - K(M_0)$, then $\gtp(a/M;N)$ forks over $M_0$.
	\item If $\pi(a) \in K(M_0)$, then there is some $a_0 \in G^*(M_0)$ and $\gamma\in G(N)$ such that
$$N \vDash a_0 + \gamma_0 = a$$
Then $\gtp(a/M; N)$ does not fork over $M_0$ if and only if $\gtp(\gamma_0/M;N)$ does not fork over $M_0$.
\end{enumerate}
\item If $a \in H^*(N)$, then $\gtp(a/M; N)$ does not fork over $M_0$ if and only if $a \in N-M$.
\end{enumerate}
\end{prop}
Note that the forking for $G^*$ and $H^*$ have identical characterizations, but since we always have $H(N) = H(M) = H(M_0)$, nonforking reduces to nonalgebraicity for $H^*$.
\begin{proof}
The proof of each case is a straightforward calculation along the lines of \cite[10.2]{ext-frame-jml}.  As an example, we show (\ref{nf-G}).

First, suppose that $\supp_N \gamma \cap M \subseteq M_0$.  Let $M_0^* \subseteq M_0$ and $M_0^* \subseteq M_\ell^* \subseteq M$ for $\ell = 0,1$ such that
\begin{itemize}
	\item $|I(M_0) - I(M_0^*)| \geq \mu$ with $\supp_N\gamma \cap M \subseteq M_0^*$; and
	\item there is $h:M_1^* \cong_{M_0^*} M_2^*$.
\end{itemize}
Then, using the extension of solutions, we can extend $h$ to an automorphism $h^+$ of $N$ such that 
\begin{itemize}
	\item if $x \in I(M_2^*)$, then $h^+(x) = h^{-1}(x)$; and
	\item if $x \in I(N) - \left(I(M_1^*) \cup I(M_2^*)\right)$, then $h^+(x) = x$.
\end{itemize}
Thus $h^+(\gamma) = \gamma$.  This shows that $h(p \rest M_1^*) = p \rest M_2^*$.  Since the $M_\ell^*$ were arbitrary, $p$ does not $\mu$-split over $M^*_0$.

Second, suppose that $\supp_N \gamma \cap M \not \subseteq M_0$ and let $\{i_1, \dots, i_r\} = \supp_N \gamma \cap \left(M - M_0\right)$.  Find $M_0^* \subseteq M_0$ containing $\supp_N \gamma \cap M_0$ such that $|I(M_0)- I(M_0^*)| = \mu$.  Then we can find $M_1^*, M_2^* \in \K^n$ such that
\begin{itemize}
	\item $M_0^* \subseteq M_\ell^* \subseteq M$;
	\item $|I(M) - I(M_\ell^*)| = |I(M_\ell^*) - I(M_0^*)| = \mu$; and
	\item $\{i_1, \dots, i_r\} \subseteq M_1^* - M_2^*$.
\end{itemize}
By the extension of solutions, there is an isomorphism $h:M_1^* \cong_{M_0^*} M_2^*$.  Then
$$N \vDash \left[\exists k \in K \left(i_1 \in k \wedge \gamma(k) = 1\right)\right] \wedge \neg \left[\exists k \in K \left(i_1 \in k \wedge \gamma(k) = 1\right)\right]$$
This witnesses that $h\left(\gtp(\gamma/M_1^*; N)\right) \neq \gtp(\gamma/M_2^*; N)$ and, thus, that $\gtp(\gamma/M; N)$ $\mu$-splits over $M_0$.
\end{proof}

Note that the case $k = 0$ is missing from Theorem \ref{abstract-positive}, and will have to be treated differently (see Theorem \ref{hs-successful} and Corollary \ref{hs-aleph0-frame}). On the negative side, we show that $\s^{n - 3, n}$ cannot be successful. First, we show that it is $\goodp$ (Definition \ref{goodp-def}).

\begin{lemma}\label{goodp-lem}
  For $n \in [3, \omega)$ and $k \le n - 3$, $\s^{k, n}$ is $\goodp$.
\end{lemma}
\begin{proof}
  Essentially this is because forking is trivial. In details, suppose that $\s^{k, n}$ is not $\text{good}^+$ and fix $\seq{M_i : i < \lambda^+}$, $\seq{N_i : i < \lambda^+}$, $\seq{a_i : i < \lambda^+}$ and $p$ witnessing it. The set of $i < \lambda^+$ such that $M_{\lambda^+} \cap N_i = M_i$ is club, so pick such an $i$. Since $\gtp (a_{i + 1} / M_{i + 1}; M_{i + 2})$ is a nonforking extension of $p$, we know that $a_{i + 1} \in I (M_{i + 2}) \backslash I (M_{i + 1})$. Because $M_{\lambda^+} \cap N_i = M_i$, we have that $a_{i + 1} \notin |N_i|$. Since $a_{i + 1} \in I (M_{i  +2})$, also $a_{i + 1} \in I (N_{i + 2})$. Therefore $\gtp (a_{i + 1} / N_i; N_{i + 2})$ does not fork over $M_0$, contradicting the defining assumption on $\seq{N_i : i < \lambda^+}$.
\end{proof}

\begin{cor}\label{not-succ-prop}
  For $n \in [3, \omega)$, $\s^{n - 3, n}$ is \emph{not} successful.
\end{cor}
\begin{proof}
  Suppose for a contradiction that $\s^{n - 3, n}$ is successful. Let $\lambda := \aleph_{n - 3}$. By Fact \ref{succ-fact}, we can get a good $\lambda^+$-frame on the saturated models of $\K_{\lambda^+}^n$. Since $\K^n$ is categorical in $\lambda^+$, this gives a good $\lambda^+$-frame on $\K_{\lambda^+}^n$. In particular, $\K^n$ is stable in $\lambda^+$, contradicting Fact \ref{bk-fact}.
\end{proof}

Notice that the proof gives no information as to which part of the definition of successful fails: i.e.\ whether $\s^{n - 3, n}$ has the existence property for uniqueness triples (and then smoothness for $\leanfn$ must fail) or not. To understand this, we take a closer look at uniqueness triples in the specific context of the Hart-Shelah example.

\section{Uniqueness triples in Hart-Shelah}

In this section, we show that the frame $\s^{n-3,n}$ is \emph{not} weakly successful.  This follows from the fact that the existence of uniqueness triples corresponds exactly to amalgamation of solutions.

The following says that it is sufficient to check one point extensions when trying to build uniqueness triples.

\begin{lemma}\label{main-red-lem}
  Let $n \in [3, \omega)$ and let $k \le n - 3$. The good $\aleph_k$-frame $\s^{k, n}$ (see Definition \ref{simple-frame-def}) is weakly successful if the following holds.

    \begin{itemize}
      \item[$(\ast)$] Whenever $M, M_a, M_b, M_{ab} \in \K^n_{\aleph_k}$ are such that:
        \begin{enumerate}
	\item $I(M_x) = I(M) \cup \{x \}$ for $x = a, b, ab$;
	\item $M \lean M_a, M_b$ and $M_b \lean M_{ab}$; and
	\item there is $f_\ell:M_a \to_M M_{ab}$ such that $f_\ell(a) = a$.
        \end{enumerate}
        Then there is $f_*: M_{ab} \cong_{M_b} M_{ab}$ such that $f_* \circ f_1 = f_2$
    \end{itemize}
\end{lemma}
\begin{remark}
  By an easy renaming exercise, we could have the range of $f_\ell$ be distinct one point extensions of $M_b$ with $f_\ell(a)$ being that point.
\end{remark}
\begin{proof}[Proof of Lemma \ref{main-red-lem}]
  Suppose that $(*)$ holds.  Let $p =\gtp(a/M;N^+) \in \gS^{bs}(M)$ and find some\footnote{$M_a$ is not unique, but there is such an $M_a$} $M_a \lean N^+$ so $I(M_a)  = I(M) \cup \{a\}$.  We want to show that this is a uniqueness triple.  To this end, suppose that we have $N \succ M$, $N \lean M_\ell$, and $f_\ell:M_a \to_M M_\ell$ with $f_\ell(a) \not\in N$.  Enumerate $I(N) - I(M) = \{ a_i \mid i < \mu \leq \aleph_k\}$; Without loss of generality $I(M_1) \cap I(M_2) = I(N)$.  Let $M_\ell^- \lean M_\ell$ be such that $I(M_\ell^-) = \{f_\ell(a)\} \cup I(N)$.

{\bf Claim:}  We can find $f^*_-:M^-_1 \cong_N M^-_2$ such that $(f^*_-)^{-1} \circ f_1 = f_2$.

This is enough: from the claim, we have $M_1^- \lean M_1$ and $f^*_-: M_1^- \to M_2$.  The class has disjoint amalgamation by Fact \ref{bk-fact}, so find a disjoint amalgam $N^*$ with maps $g_\ell:M_\ell \to N^*$ such that $g_1 \rest M_1^- = g_2 \circ f^*_-$.  This is the witness required to have that $(a, M, M_a)$ is a uniqueness triple.

{\bf Proof of the claim:}  We can find resolutions $\seq{N_i : i < \mu}$ and $\seq{M^\ell_i \mid i < \mu}$ such that:

\begin{enumerate}
	\item $M \lean N_i \lean M^\ell_i \lean M_\ell^-$ and $f_\ell(M_a) \lean M_i^\ell$; and
	\item $I(N) = I(M) \cup \{ a_j \mid j<i\}$ and $I(M_i^\ell) = I(N_i) \cup \{f_\ell(a)\}$.
\end{enumerate}
The values of $I$ for these models is specified, which determines $K$ and $G$.  Then $G^*$ and $H^*$ are just picked to be subsets of the larger models version that is closed under the relevant action.  Since there are embeddings going everywhere, this can be done.

We build increasing, continuous $f_i^*:M_i^1 \cong_{N_i} M_i^2$ such that $f^*_i \circ f_1 = f_2$ by induction on $i \geq 1$.
\begin{itemize}
	\item For $i=1$, we use $(*)$ taking $b = a_0$ (and using the renamed formulation).  This gives $f_1^*:M_1^1 \cong_{N_1} M_1^2$.
	\item For $i$ limit, we take unions of everything.
	\item For $i = j+1$, we have an instance of $(*)$:
	\[
\xymatrix{ & M^1_{j+1} & \\
M_j^1 \ar[ur] \ar[rr]_>>>>>>>{f^*_i} & & M^2_{j+1}\\
N_j \ar[r] \ar[u] & N_{j+1} \ar[uu] \ar[ur] & }
\]
	Then we can find $f_{i+1}^*:M^1_{j+1} \cong M^2_{j+1}$ that works. 
\end{itemize}
\end{proof}

We can now give a direct proof of Theorem \ref{abstract-positive} that also treats the case $k = 0$. 

\begin{theorem}\label{hs-successful}
  Let $n \in [3, \omega)$. For any $k < n - 3$, $\s^{k, n}$ is successful.
\end{theorem}
\begin{proof}
  By Fact \ref{succ-criteria} (as in the proof of Theorem \ref{abstract-positive}), it is enough to show that $\s^{k, n}$ is weakly successful. It suffices to show $(\ast)$ from Lemma \ref{main-red-lem}.  We start with a solution $h$ on $I(M)$.  Working inside $M_{ab}$, we can find extensions $h^1_a, h^2_a, h_b$ of $h$ that are solutions for $f_1(M_a), f_2(M_a), M_b$ by the extension property of solutions (which holds because 2-amalgamation does).  Now, for $\ell = 1, 2$, amalgamate $h^\ell_a$ and $h_b$ over $h$ into $h^\ell_{ab}$, which is a solution for $M_{ab}$.  We use this to get a isomorphism $f_*$.

Set $f_*$ to be the identity on $I(M_{ab}) = I(M) \cup \{a, b\}$.  This determines its value on $K$, $G$, and $\mathbb{Z}_2$.  

Let $x \in G^*_u(M_{ab})$ for $u \in K(M_{ab})$.  There is a unique $\gamma \in G(M_{ab})$ such that $t_{G^*}^{M_{ab}}(u, f^1_{ab}(u), x, \gamma)$.  Then, there is a unique $y \in G^*_u(M_{ab})$ such that $t_{G^*}^{M_{ab}}(u, f^2_{ab}(u), y, \gamma)$.  Set $f_*(x) = y$.

Let $x \in H^*_u(M_{ab})$ for $u \in K(M_{ab})$.  There is a unique $n \in H(M_{ab})$ such that $t_{H^*}^{M_{ab}}(u, f^1_{ab}(u), x, n)$.  Then, there is a unique $y \in H^*_u(M_{ab})$ such that $t_{H^*}^{M_{ab}}(u, f^2_{ab}(u), y, n)$.  Set $f_*(x) = y$.

This is a bijection on the universes, and clearly preserves all structure except maybe $Q$.  So we show it preserves $Q$.  It suffices to show one direction for positive instances of $Q$.  So let $u_1, \dots, u_k, v$ be compatible from $K(M_{ab})$ and $x_j \in G^*_{u_j}(M_{ab}), y \in H^*_v(M_{ab})$ such that
$$M_{ab} \vDash Q(x_1,\dots, x_k, y)$$
Note, by definition of solutions, we have
$$M_{ab} \vDash Q\left(f_{ab}^1(u_1), \dots, f_{ab}^1(u_k), g^1_{ab}(v)\right)$$
$$M_{ab} \vDash Q\left(f_{ab}^2(u_1), \dots, f_{ab}^2(u_k), g^2_{ab}(v)\right)$$
By the properties of $Q$, we get $\gamma_j \in G(M_{ab})$ and $n \in H(M_{ab})$ such that
\begin{enumerate}
	\item $t_{G^*}^{M_{ab}}(u_j, f^1_{ab}(u_j), x_j, \gamma_j)$
	\item $t_{H^*}^{M_{ab}}(v, g^1_{ab}(v), y, n)$
	\item $\gamma_1(v) + \dots + \gamma_k(v) + n \equiv 0 \mod 2$
\end{enumerate}
Then, by definition of $f_*$, we have
\begin{enumerate}
	\item $t_{G^*}^{M_{ab}}(u_j, f^2_{ab}(u_j), f_*(x_j), \gamma_j)$
	\item $t_{H^*}^{M_{ab}}(v, g^2_{ab}(v), f_*(y), n)$
\end{enumerate}
By the evenness of these shifts, we have that
$$M_{ab} \vDash Q\left(f_*(x_1),\dots, f_*(x_k), f_*(y)\right)$$
Perfect.

The commutativity condition is easy to check.
\end{proof}

The next two lemmas show that the uniqueness triples (if they exist) must be exactly the one point extensions. This can be seen from the abstract theory \cite[III.3.5]{shelahaecbook} but we give a direct proof here.

\begin{lemma} \label{uniq-trip-cor-2}\label{uniq-trip-cor}
  Let $n \in [3, \omega)$ and let $k \le n - 3$. If $(a, M, M^+)$ is a uniqueness triple of $\s^{k, n}$, then $I(M^+) = I(M) \cup \{a \}$.
\end{lemma}

Recall (Definition \ref{standard-def}) that the standard model is the one where $G^*$ is literally equal to $K \times G$, so that we can easily recover $0$'s.

\begin{proof}
  Deny.  By Lemma \ref{ext-imp-stand}, without loss of generality, we have that $M$ is the standard model on $I(M) = X$ and $M^+$ is the standard model on $I(M^+) = X \cup X^+ \cup \{a\}$ (those unions are disjoint) with $X^+$ nonempty.  Set $N$ to be the standard model on $X \cup (2 \times X^+)$ and $N_0, N_1$ to be standard models on $X \cup 2 \times X^+ \cup \{a\}$.  For $\ell = 0,1$, define $f_\ell: M^+ \to_M N_\ell$ by 
\begin{enumerate}
	\item $f_\ell$ is the identity on $X \cup \{a\}$ and sends $x \in X^+$ to $(\ell, x)$.
	\item The above determines the map on $K$, $H$, and $G$.
	\item $(u, x) \in G^*(M^+)$ goes to $(f_\ell(u), x) \in G^*(N_\ell)$.
	\item $(u, n) \in H^*(M^+)$ goes to $(f_\ell(u), n) \in H^*(N_\ell)$.
\end{enumerate}
Then this is clearly a set-up for weak uniqueness.  However, suppose there were a $N^*$ with $g_\ell:N_\ell \to_N N^*$ such that $g_0 \circ f_0 = g_1 \circ f_1$.  Let $x \in X^+$.  Then
$$(0, x) = g_0(x) = f_0(g_0(x)) = f_1(g_1(x)) = f_1(1,x) = (1,x)$$
which is false.
\end{proof}

\begin{lemma} \label{uniq-trip-char}
  Let $n \in [3, \omega)$ and let $k \le n - 3$. Let $M \lean N$ both be in $\K_{\aleph_k}^n$. If $\s^{k, n}$ is weakly successful, then $(a, M, N)$ is a uniqueness triple of $\s^{k, n}$ if and only if $I(N) = I(M) \cup \{a\}$.
\end{lemma}
\begin{proof}
  Lemma \ref{uniq-trip-cor} gives one direction. Conversely, let $(a, M, N)$ with $I(N) = I(M) \cup \{a\}$.  Since $\s^{k, n}$ is weakly successful, there is some uniqueness triple $(b, M', N')$ representing $\gtp (a /M; N)$. By Lemma \ref{uniq-trip-cor}, we must have $I(N') = I(M') \cup \{b\}$.  By Lemma \ref{ext-imp-stand}, we have $(M, N) \cong (M', N')$ since they are both isomorphic to the standard model.  This isomorphism must take $a$ to $b$.  Since $(a, M, N) \cong (b, M', N')$, the former is a uniqueness triple as well.
\end{proof}

We deduce that $\s^{n - 3, n}$ is not even weakly successful.

\begin{theorem}\label{negative-prop}
  For $n \in [3, \omega)$, $\s^{n - 3, n}$ is not weakly successful.
\end{theorem}
\begin{proof}
  Let $\lambda := \aleph_{k - 3}$. At this cardinal, 2-amalgamation of solutions over sets of size $\lambda$ fails.  To witness this, we have:
\begin{itemize}
	\item $M$ of size $\lambda$ with solution $h = (f, g)$
	\item $M_a$ has a solution $h_a = (f_a, g_a)$
	\item $M_b$ has a solution $h_b = (f_b, g_b)$
	\item $M_{ab}$ has no solution that extends them both
	\item $I(M_x) = I(M) \cup \{x\}$ for $x = a, b, ab$
\end{itemize}
However, $\lambda$ does have extension of solutions, so let $h_{ab} = (f_{ab}, g_{ab})$ be a solution for $M_{ab}$ that extends $h_b$.  $h_{ab}$ is a solution for $I(M_a)$ in $M_{ab}$.\footnote{Note that it isn't a solution in $M_a$ as $f_{ab}(u)$ might not be in $M_a$ for $u \in M_a$.}  Set $f_1:M_a \to_{M} M_{ab}$ to be the identity.  Define $f_2: M_a \to_M M_{ab}$ as follows:
\begin{itemize}
	\item identity on $I(M) \cup\{a\}$, which determines it except on the affine stuff (in the sense of Lemma \ref{not-aff-ext})
	\item Let $x \in G^*_u(M_a)$ for $u \in K(M_a)$.  Set $f_2$ to send $f_a(u)$ to $f_{ab}(u)$ and the rest falls out by the $G$ action
	\item Let $x \in H^*_u(M_a)$ for $u \in K(M_a)$.  Set $f_2$ to send $g_a(u)$ to $g_{ab}(u)$ and the rest falls out by the $G$ action.
\end{itemize}
This map commutes on $M$ because if $u \in K(M)$, then $f_{ab}(u) = f_a(u) = f(u)$ .

We claim that $\gtp(a/M; M_a)$ does not have a uniqueness triple.  Suppose it does. By Lemma \ref{uniq-trip-char}, $(a, M, M_a)$ is one.

Suppose that we had $N^*$ and $g_\ell:M_{ab} \to_{M_b} N^*$ such that $g_1 = g_2f_2$ and $g_1(a) = g_2(f_2(a))$ (recalling that $f_1$ is the identity).

{\bf Claim:} If $u \in K(M_a)$, then $g_1(G^*_u(M_{ab})) = g_2(G^*_u(M_{ab}))$.\\

There is $\gamma_u \in G(M_{ab})$ such that $f_{ab}(u) = f_{a}(u) + \gamma_u$.  Given $x \in G^*_u(M_{ab})$,
$$g_1(x) = g_2(f_2(x)) = g_2(x + \gamma_u) = g_2(x) + \gamma_u$$
Thus $g_1(G^*_u(M_{ab}))$ and $g_2(G^*_u(M_{ab}))$ are both subsets of $G^*_u(N^*)$ that have a 1-transitive action of $G(M_{ab})$ and share points.\hfill $\dag_{\text{Claim}}$\\

Now define $h^+=(f^+, g^+)$ on $M_{ab}$ by
\begin{eqnarray*} 
f^+(u) &=& g_1^{-1} \circ g_2 \circ f_{ab}(u)\\
g^+(u) &=& g_1^{-1} \circ g_2 \circ g_{ab}(u)
\end{eqnarray*}
We claim $h^+$ extends both $h_a$ and $h_b$.  If $u \in K(M_b)$, then $f_{ab}(u) = f_b(u) \in M_b$, so
$$f^+(u) = g_1^{-1}\circ g_2 \circ f_{ab}(u) = f_{ab}(u) = f_b(u)$$
since the $g_\ell$'s fix $M_b$.  Suppose $u \in K(M_a)$.  First note that $g_1^{-1} \circ g_2 = f_2^{-1}$ by assumption.  Also, since $f_2(f_a(u)) = f_{ab}(u)$ and $f_2$ respects the group action, $f_2(f_{ab}(u)) = f_a(u)$.  Thus
$$f^+(u) = g_1^{-1} \circ g_2 \circ f_{ab}(u) = f_2^{-1} \circ f_{ab}(u) = f_a(u)$$
Similarly for $g^+$.

But this is our contradiction! $h_a$ and $h_b$ were not amalgamable, so there is no isomorphism.
\end{proof}

\section{Nonforking is disjoint amalgamation}

Recall that if a good frame is weakly successful, one can define an independence relation $\NF$ for models (see Definition \ref{nf-def}). We show here that $\NF$ in the Hart-Shelah example is just disjoint amalgamation, i.e.\ $\NF (M_0, M_1, M_2, M_3)$ holds if and only if $M_0 \lean M_\ell \lean M_3$ for $\ell < 4$ and $M_1 \cap M_2 = M_0$. We deduce another proof of Theorem \ref{negative-prop}.

We will use the following weakening of \cite[4.2]{bk-hs}

\begin{fact} \label{jep-in-model-lemma-2}
  Let $n \in [2, \omega)$. If $M_0, M_1 \lean N$, then there is $M_2 \lean N$ such that $I(M_2) = I(M_0) \cup I(M_1)$ and $M_0, M_1 \lean M_2$.
\end{fact}

\begin{theorem}\label{nf-dap}
  Let $n \in [3, \omega)$ and let $k \le n - 3$. Let $\lambda := \aleph_k$ and let $M_0, M_1, M_2, M_3 \in \K_\lambda^n$ with $M_0 \lean M_\ell \lean M_3$ for $\ell < 4$. If $\s^{k, n}$ is weakly successful, then $\NF_{\s^{k, n}} (M_0, M_1, M_2, M_3)$ if and only if $M_1 \cap M_2 = M_0$.
\end{theorem}
\begin{proof}
  Write $\NF$ for $\NF_{\s^{k, n}}$. The left to right direction follows from the properties of $\NF$ (Fact \ref{nf-fact}). Now assume that $M_1 \cap M_2 = M_0$. 

  Write $I(M_1) - I(M_0) = \{d_i \mid i < \alpha^\ast\}$.  By induction, build increasing, continuous $M_{1, i}  \lean M_1$ for $i < \alpha^\ast$ so $I(M_{1, i}) = I(M_0) \cup \{d_j \mid j < i\}$.  Again by induction, build increasing continuous $M_{2, i} \lean M_3$ for $i \le \alpha^\ast$ such that
\begin{itemize}
	\item $I(M_{2, i}) = I(M_2) \cup \{d_j \mid j < i\}$
	\item $M_{1, i} \lean M_{2, i}$
\end{itemize}
The successor stage of this construction is possible by Fact \ref{jep-in-model-lemma-2} and the limit is easy.  Now it's easy to see that $\gtp(d_i/M_{2, i}; M_{2, i+1})$ does not fork over $M_{1, i}$. Furthermore by Lemma \ref{uniq-trip-char}, $(d_i, M_{1, i}, M_{1, i + 1})$ is a uniqueness triple. Thus letting $M_3' := M_{2, \alpha^\ast}$,  we have that $\NF^\ast (M_0, M_1, M_2, M_3')$, so $\NF (M_0, M_1, M_2, M_3')$. By the monotonicity property of $\NF$, $\NF (M_0, M_1, M_2, M_3)$ also holds.
\end{proof}

We deduce another proof of Theorem \ref{negative-prop}. First we show that weakly successful implies successful in the context of Hart-Shelah:

\begin{lemma}\label{weakly-succ-lem}
  Let $n \in [3, \omega)$ and let $k \le n - 3$. If $\s^{k, n}$ is weakly successful, then $\s$ is successful (recall Definition \ref{succ-def}). Moreover for $M_0, M_1 \in \K_{\lambda^+}^n$, $M_0 \leanfn M_1$  if and only if $M_0 \lean M_1$.
\end{lemma}
\begin{proof}
  This is straightforward from Definition \ref{succ-def} and Theorem \ref{nf-dap}.
\end{proof}

\begin{cor}\label{negative-prop-2}
  For $n \in [3, \omega)$, $\s^{n - 3, n}$ is not weakly successful.
\end{cor}
\begin{proof}
  Assume for a contradiction that $\s^{n - 3, n}$ is weakly successful. By Lemma \ref{weakly-succ-lem}, $\s^{n - 3, n}$ is successful. This contradicts Corollary \ref{not-succ-prop}. 
\end{proof}

\section{A type-full good frame at $\aleph_0$}\label{good-frame-sec}

We have seen that when $k < n -3$, $\s^{k, n}$ is successful $\goodp$ and therefore by Fact \ref{type-full-succ} extends to a type-full frame. When $k = n - 3$, $\s^{k, n}$ is not successful, but by Theorem \ref{abstract-positive}, it still extends to a type-full frame if $k \ge 1$. In this section, we complete the picture by building a type-full frame when $k = 0$ and $n = 3$.

Recall that (when $n \ge 3$) $\K^n$ is a class of models of an $\Ll_{\omega_1, \omega}$ sentence, categorical in $\aleph_0$ and $\aleph_1$. Therefore by \cite[II.3.4]{shelahaecbook} (a generalization of earlier results in \cite{sh48, sh87a}), there will be a good $\aleph_0$-frame on $\K^n$ \emph{provided} that $2^{\aleph_0} < 2^{\aleph_1}$. Therefore the result we want is at least consistent with ZFC, but we want to use the additional structure of the Hart-Shelah example to remove the cardinal arithmetic hypothesis.

So we take here a different approach than Shelah's, giving new cases on when an AEC has a good $\aleph_0$-frame. As opposed to Shelah, we use Ehrenfeucht-Mostowski models (so assume that the AEC has arbitrarily large models).

Shelah has defined the following property \cite[1.3(2)]{shelahaecbook}\footnote{Shelah defines saturative as a property of frames, but it depends only on the class.}:

\begin{defin}\label{saturative-def}
  $\K$ is \emph{$\lambda$-saturative} (or \emph{saturative in $\lambda$}) if for any $M_0 \lea M_1 \lea M_2$ all in $\K_\lambda$, if $M_1$ is limit over $M_0$, then $M_2$ is limit over $M_0$.
\end{defin}

An immediate consequence of Theorem \ref{universal-char-claim} is:

\begin{cor} \label{limit-char-claim}
  Let $n \in [3, \omega)$. For any $k \le n - 3$, $\K^n$ is saturative in $\aleph_k$.
\end{cor}

We will use the following consequence of being saturative:

\begin{lemma}\label{saturative-chain}
  Assume that $\LS (\K) = \aleph_0$, and $\K_{\aleph_0}$ has amalgamation, no maximal models, and is stable in $\aleph_0$. Let $\seq{M_i : i \le \omega}$ be an increasing continuous chain in $\K_{\aleph_0}$. If $\K$ is categorical in $\aleph_0$ and saturative in $\aleph_0$, then there exists an increasing continuous chain $\seq{N_i : i \le \omega}$ such that:

  \begin{enumerate}
  \item\label{cond-1} For $i < \omega$, $M_i$ is limit over $N_i$.
  \item\label{cond-2} For $i < \omega$, $N_{i + 1}$ is limit over $N_i$.
  \item $N_\omega = M_\omega$.
  \end{enumerate}
\end{lemma}
\begin{proof}
  Let $\{a_n : n < \omega\}$ be an enumeration of $|M_\omega|$. We will build $\seq{N_i : i \le \omega}$ satisfying (\ref{cond-1}) and (\ref{cond-2}) above and in addition that for each $i < \omega$, $\{a_n : n < i\} \cap |M_{i}| \subseteq |N_{i}|$. Clearly, this is enough.

  This is possible. By categoricity in $\aleph_0$, any model of size $\aleph_0$ is limit, so pick any $N_0 \in \K_{\aleph_0}$ such that $M_0$ is limit over $N_0$. Now assume inductively that $N_i$ has been defined for $i < \omega$. Since $\K$ is saturative in $\aleph_0$, $M_{i + 1}$ is limit over $N_i$. Since all limit models of the same cofinality are isomorphic, $M_{i + 1}$ is in particular $(\aleph_0, \omega \cdot \omega)$-limit over $N_i$. Fix an increasing continuous sequence $\seq{M_{i + 1, j} : j  \le \omega \cdot \omega}$ witnessing it: $M_{i + 1, 0} = N_i$, $M_{i + 1, \omega \cdot \omega} = M_{i + 1}$, and $M_{i + 1, j + 1}$ is universal over $M_{i + 1, j}$ for all $j < \omega \cdot \omega$. Now pick $j < \omega \cdot \omega$ big enough so that $\{a_n : n < i + 1\} \cap |M_{i + 1}| \subseteq |M_{i + 1, j}|$. Let $N_{i + 1} := M_{i + 1, j + \omega}$.
\end{proof}
\begin{remark}
  We do not know how to replace $\aleph_0$ by an uncountable cardinal in the argument above: it is not clear what to do at limit steps.
\end{remark}

To build the good frame, we will also use the transitivity property of splitting:

\begin{defin}
  We say that $\K$ satisfies \emph{transitivity in $\mu$} (or \emph{$\mu$-transitivity}) if whenever $M_0, M_1, M_2 \in \K_{\mu}$, $M_1$ is limit over $M_0$ and $M_2$ is limit over $M_1$, if $p \in \gS (M_2)$ does not $\mu$-split over $M_1$ and $p \rest M_1$ does not $\mu$-split over $M_0$, we have that $p$ does not $\mu$-split over $M_0$.
\end{defin}

The following result of Shelah \cite[7.5]{sh394} is key:

\begin{fact}\label{trans-from-solvability}
  Let $\mu \ge \LS (\K)$. Assume that $\K_\mu$ has amalgamation and no maximal models. If $\K$ has arbitrarily large models and is categorical in $\mu^+$, then $\K$ has transitivity in $\mu$.
\end{fact}

We will also use two lemmas on splitting isolated by VanDieren \cite[I.4.10, I.4.12]{vandierennomax}.

\begin{fact}\label{ns-facts}
  Let $\mu \ge \LS (\K)$. Assume that $\K_\mu$ has amalgamation, no maximal models, and is stable in $\mu$. Let $M_0 \lea M \lea N$ all be in $\K_\mu$ such that $M$ is universal over $M_0$.

  \begin{enumerate}
  \item Weak extension: If $p \in \gS (M)$ does not $\mu$-split over $M_0$, then there exists $q \in \gS (N)$ extending $p$ and not $\mu$-splitting over $M_0$. Moreover $q$ is algebraic if and only if $p$ is algebraic.
  \item Weak uniqueness: If $p, q \in \gS (N)$ do not $\mu$-split over $M_0$ and $p \rest M = q \rest M$, then $p = q$.
  \end{enumerate}
\end{fact}

We are now ready to build the good frame:

\begin{theorem}\label{good-frame-thm}
  If:

  \begin{enumerate}
  \item $\K$ is superstable in $\aleph_0$.
  \item $\K$ has symmetry in $\aleph_0$.
  \item $\K$ has transitivity in $\aleph_0$.
  \item $\K$ is categorical in $\aleph_0$.
  \item $\K$ is saturative in $\aleph_0$.    
  \end{enumerate}

  Then there exists a type-full good $\aleph_0$-frame with underlying class $\K_{\aleph_0}$.
\end{theorem}
\begin{proof}
  By the superstability assumption, $\K_{\aleph_0}$ has amalgamation and no maximal models and is stable in $\aleph_0$. By the categoricity assumption, $\K_{\aleph_0}$ also has joint embedding. It remains to define an appropriate forking notion. For $M \lea N$ both in $\K_{\aleph_0}$, let us say that $p \in \gS (N)$ \emph{does not fork over $M$} if there exists $M_0 \in \K_{\aleph_0}$ such that $M$ is universal over $M_0$ and $p$ does not $\aleph_0$-split over $M_0$. We check that it has the required properties (see Definition \ref{good-frame-def}):

  \begin{enumerate}
  \item Invariance, monotonicity: Straightforward.
  \item Extension existence: By the weak extension property of splitting (Fact \ref{ns-facts}).
  \item Uniqueness: Let $M \lea N$ both be in $\K_{\aleph_0}$ and let $p, q \in \gS (N)$ be nonforking over $M$ such that $p \rest M = q \rest M$. Using the extension property, we can make $N$ bigger if necessary to assume without loss of generality that $N$ is limit over $M$. By categoricity, $M$ is limit. Pick $\seq{M_i : i \le \omega}$ increasing continuous witnessing it (so $M_\omega = M$ and $M_{i + 1}$ is universal over $M_i$ for all $i < \omega$). By the superstability assumption, there exists $i < \omega$ such that $p \rest M$ does not $\aleph_0$-split over $M_i$ and there exists $j < \omega$ such that $q \rest M$ does not $\aleph_0$-split over $M_j$. Let $i^\ast := i + j$. Then both $p \rest M$ and $q \rest M$ do not $\aleph_0$-split over $M_{i^\ast}$. By $\aleph_0$-transitivity, both $p$ and $q$ do not $\aleph_0$-split over $M_{i^\ast}$. Now use the weak uniqueness property of splitting (Fact \ref{ns-facts}).
  \item Continuity: In the type-full context, this follows from local character (see \cite[II.2.17(3)]{shelahaecbook}).
  \item Local character: Let $\delta < \omega_1$ be limit and let $\seq{M_i : i \le \delta}$ be increasing continuous in $\K_{\aleph_0}$. Let $p \in \gS (M_\delta)$. We want to see that there exists $i < \delta$ such that $p$ does not fork over $M_i$. We have that $\cf{\delta} = \omega$, so without loss of generality $\delta = \omega$. Let $\seq{N_i : i \le \omega}$ be as given by Lemma \ref{saturative-chain} (we are using saturativity here). By superstability, there exists $i < \omega$ such that $p$ does not $\aleph_0$-split over $N_i$. Because $M_i$ is limit (hence universal) over $M_i$, this means that $p$ does not fork over $M_i$, as desired.
  \item Symmetry: by $\aleph_0$-symmetry (see \cite[4.12]{vv-symmetry-transfer-afml}).
  \end{enumerate}
\end{proof}
\begin{cor}\label{good-frame-cor}
  Assume that $\LS (\K) = \aleph_0$. If:

  \begin{enumerate}
  \item $\K$ has amalgamation in $\aleph_0$.
  \item $\K$ is categorical in $\aleph_0$.
  \item $\K$ is saturative in $\aleph_0$.
  \item $\K$ has arbitrarily large models and is categorical in $\aleph_1$.
  \end{enumerate}

  Then there exists a type-full good $\aleph_0$-frame with underlying class $\K_{\aleph_0}$.
\end{cor}
\begin{proof}
  It is enough to check that the hypotheses of Theorem \ref{good-frame-thm} are satisfied. First note that $\K$ has no maximal models in $\aleph_0$ because it has a model in $\aleph_1$ (by solvability) and is categorical in $\aleph_0$. Therefore by Fact \ref{shvi}, $\K$ is $\aleph_0$-superstable. By Fact \ref{sym-categ}, $\K$ has $\aleph_0$-symmetry. Finally by Fact \ref{trans-from-solvability}, $\K$ has $\aleph_0$-transitivity.
\end{proof}

\begin{cor}\label{hs-aleph0-frame}
  For $n \in [3, \omega)$, there exists a type-full good $\aleph_0$-frame on $\K^n$.
\end{cor}
\begin{proof}
  By Fact \ref{bk-fact} and Corollary \ref{limit-char-claim}, $\K^n$ satisfies the hypotheses of Corollary \ref{good-frame-cor}.
\end{proof}

The argument also allows us to prove that Theorem \ref{shortness-thm} is optimal, even when $n = 3$:

\begin{cor}\label{type-short-negative}
  For $n \in [3, \omega)$, $\K^n$ is not $(<\aleph_{n - 3}, \aleph_{n - 3})$-type short over $\aleph_{n - 3}$-sized models.
\end{cor}
\begin{proof}
  Let $\lambda := \aleph_{n - 3}$. By Theorem \ref{abstract-positive} (or Corollary \ref{hs-aleph0-frame} if $\lambda = \aleph_0$), there is a type-full good $\lambda$-frame $\s$ on $\K_\lambda$. Assume for a contradiction that $\K^n$ is $(<\lambda, \lambda)$-type short over $\lambda$-sized models. We will prove that $\s$ is weakly successful. This will imply (by Fact \ref{canon-fact} and the definition of uniqueness triples) that $\s^{n - 3, n}$ is weakly successful, contradicting Theorem \ref{negative-prop}. First observe that by Theorem \ref{shortness-thm}, $\K^n$ must be $(<\aleph_0, \lambda)$-type short over $\lambda$-sized models.

  We now consider two cases.

  \begin{itemize}
  \item If $\lambda > \aleph_0$, then (recalling Facts \ref{bk-fact} and \ref{canon-fact}) by Fact \ref{good-frame-constr} (where $\kappa$ there stands for $\aleph_0$ here), $\s$ is weakly successful, which is the desired contradiction.
  \item If $\lambda = \aleph_0$, we proceed similarly: For $M \lea N$ both in $\K_{\aleph_0}$ and $p \in \gS^{\alpha} (N)$ with $\alpha < \aleph_1$, let us say that $p$ \emph{does not fork over $M$} if for every finite $I \subseteq \alpha$ there exists $M_0 \lea M$ with $M$ universal over $M_0$ such that $p^I$ does not $\mu$-split over $M_0$. As in the proof of Theorem \ref{good-frame-thm} (noting that in Fact \ref{trans-from-solvability} transitivity holds for any type of finite length), this nonforking relation has the uniqueness property for types of finite length. By the shortness assumption, it has it for types of length at most $\aleph_0$ too. It is easy to see that nonforking satisfies local character for $(<\aleph_0)$-length types over $(\aleph_0, \aleph_1)$-limits and has the left $(<\aleph_0)$-witness property (see \cite[3.7]{downward-categ-tame-apal}). Therefore by \cite[3.8, 3.9]{downward-categ-tame-apal} it reflects down (see \cite[3.7(3)]{downward-categ-tame-apal}). By \cite[3.11]{downward-categ-tame-apal}, $\s$ is weakly successful, as desired.
  \end{itemize}
\end{proof}

\bibliographystyle{amsalpha}
\bibliography{counterexamples-frames}

\end{document}